\documentclass[12pt]{article}
\usepackage{srcltx}
\usepackage{amssymb,amsmath,amsfonts,amsthm}
\usepackage{cite}
\newcounter{parag}

\newtheorem{lem}{Lemma}
\newtheorem{theorem}{Theorem}

\newtheorem{cor}{Corollary}
\newtheorem{prop}{Proposition}

\begin{document}

\begin{center}
{\bf \Large On Thompson's conjecture for finite simple groups}

{\bf I. B. Gorshkov}
\medskip
\footnote{The work was supported by
RFBR 18-31-00257.}

\end{center}
{\it Abstract: Let $G$ be a finite group, $N(G)$ be the set of conjugacy classes of the group $G$. In the present paper it is proved $G\simeq L$ if $N(G)=N(L)$, where $G$ is a finite group with trivial center and $L$ is a finite simple group.
\smallskip

 \smallskip
}

\section*{Introduction}

Consider a finite group $G$. For $g\in G$, denote by $g^G$ the conjugacy class of $G$ containing $g$, and by $|g^G|$ the size of $g^G$. The centralizer of $g$ in $G$ is denoted by $C_G(g)$. Put $N(G)=\{n\  |\  \exists g\in G$ such that $|g^G|=n\}\setminus\{1\}$. In
1987 Thompson posted the following conjecture concerning $N(G)$.
\medskip

\textbf{Thompson's Conjecture (see \cite{Kour}, Question 12.38)}. {\it If $L$ is a finite
simple non-abelian group, $G$ is a finite group with trivial
center, and $N(G)=N(L)$, then $G\simeq L$.}

\medskip

A more general question can also be formulated. When the property of equality sets of conjugacy classes sizes of two groups with a trivial center implies an isomorphism.

We say that the group $L$ is recognizable by the set of conjugacy classes size among finite groups with a trivial center (briefly recognizable) if the equality $N(L)=N(G)$, where $G$ is a group with trivial center, implies an isomorphism $L\simeq G$. Obviously, any group having a non-trivial center is not recognizable. However, it is easy to show that $S_3$ is recognizable. Thus, the condition of solvability is not a necessary condition for recognizability. As an example of an non-recognizeble group, we can take a Frobenius group of order $18$. There exists two non-isomorphic Frobenius groups of order $18$ and they have the same sets of conjugacy classes sizes. The question of the existence of the group $L$ such that there exists an infinite set of groups $\{G_i, i \in\mathbb{N} \} $ with trivial center such that $N(L)=N(G_i)$, where $i \in \mathbb{N}$, is open. The set $N(L)$ is closely connected with the order of the group $L$, and sometimes precisely determines it, which essentially limits the possibilities for constructing an infinite series of groups with a trivial center and the same set of conjugacy classes size.

We denote by $\omega(G)$ and $\pi(G)$ the set of elements orders in $G$ and the set of all prime divisors of orders of $G$ respectively. The set $\omega(G)$ defines a prime graph $GK(G)$, whose vertex set is $\pi(G)$ and two distinct primes $p, q\in\pi(G)$ are adjacent if $pq\in\omega(G)$. The greatest power of a prime $p$ dividing the natural number $n$ will be denoted by $n_p$.
The number $p^n$ such that there exists $\alpha \in N(G)$ is a multiple of $p^n$ and there is no $\beta \in N(G)$ that is a multiple of $p^{n + 1}$ is denoted by $|G||_p $. If $\pi\subseteq\pi(G)$ then $|G||_{\pi}=\prod_{p\in \pi}|G||_p $. To shorten the notation, we introduce the following notation $|G||=|G||_{\pi (G)}$. Note that the definitions of the number $|G||$ are correct and $(|G||)_p=|G||_p$. We have $|G||_p$ divides $|G|_p $ for any $ p\in\pi(G)$.

Take a group $L$ with trivial center and disconnected prime graph $GK(L)$. In \cite{Ch96} it is proved that if $N(L)=N(G)$ for any group $G$ with trivial center then $|G|=|G||=|L|$. However, in the general case the equality $|L|=|L||$ may not hold. An example of a group $G$ with a trivial center such that $|G|>|G||$ is constructed in \cite {camin11}.

The question of recognizability is most interesting for simple groups. One of the first groups for which the proof of recognizability was the sporadic group whose prime graph has more than two connected components, see \cite{Chen92}. In \cite{Shi}, using previous results, it is proved that any sporadic group is recognizable. In \cite{Shi2} it is proved that almost sporadic groups are also recognizable.
In \cite{Wil} end \cite{Kon} obtained the classification of finite simple groups with disconnected prime graph. Using this deep result Chen \cite{Ch96,Ch99} established the Thompson's conjecture for all finite simple groups with prime graph have more then two connected components. In the last 10 years, the recognition of a large number of simple groups. However, the results were of a partial nature. Until 2009, the prime graph of any group for which the recognition was proven was disconnected. In \cite {VasT} was shown the recognition of groups with a connected prime graph, in particular $Alt_ {10}$ and $L_4 (4)$ are recognizable. In \cite {Ah11}, the proof of the validity of the Thompson's conjecture for all simple Lie groups of type $A_n (q)$ is completed.
In article \cite{Xu14} showed that Thompson's conjecture holds for finite simple exceptional groups of type $E_7(q)$. In \cite{GorandCo} showed that Thompson's conjecture holds for $F_4(q)$ for odd $q, E_6(q), ^2E_6(q)$. As corollary of thats resaults we obtain that Thompson's conjecture has been proved valid for all exceptional group of lie type. In seryes articles \cite{AD, VasT, Gor0, GorA, GorA2, GorA3, GorA4, Iran, Kh1, Kh2} it was proved that Thompson's conjecture holds for all simple alternating groups. Recognizability of simple groups $^2A_n(q)$ is proved in \cite{Ah17}. In \cite {Ah16} and \cite{Ah12}, the recognizability of the groups $B_n(q)$ and $C_n(q)$ is proved. The prime graph$\ ^2B_2$ has $4$ connected components and recognition of that's groups was proved in \cite{Ch96}. In \cite{Darafsheh}, the recognizability of the groups $D_{p + 1}(2)$ and $D_{p + 1}(3)$ is proved, where $ p $ is a prime. In \cite{Ah132}, the Thompson's conjecture was proved for groups of type $D_n(q)$ where $n \not \in \{4, 8\}$. The validity of the conjecture for the groups $^2D_n(q)$ was proved in \cite{Ah13}. The graph of prime numbers of the groups $^3D_4(q) $ is disconnected, the validity of the conjecture for these groups was proved in \cite{Ch01}. Thus, to complete the proof of the validity of Thompson's conjecture, it is necessary to prove the recognizability of the groups $D_4(q)$ and $D_8(q)$ where $q>3$.

\begin{theorem}
The groups $D_4(q)$ and $D_8(q) $ are recognizable.
\end{theorem}

As a corollary of the theorem and the results obtained earlier, we obtain the following assertion.

\begin{cor}
Thompson's conjecture is hold.
\end{cor}

\section{Notation and preliminary results}

If $n$ is a nonzero integer and $r$ is an odd prime with $(r,n)=1$, then $e(r, n)$ denotes the multiplicative order of $n$ modulo $r$.
Fix an integer $a$ with $|a|>1$. A prime $r$ is said to be a primitive prime divisor of $a^i-1$ if $e(r,a)=i$. We write $r_i(a)$ to denote some primitive prime divisor of $a^i-1$, if such a prime exists, and $R_i(a)$ to denote the set of all such divisors. Zsigmondy \cite{zs} proved that primitive prime divisors exist for almost all pairs $(a, i)$.

\begin{lem}(Zsigmondy).
  Let $a$ be an integer and $|a|>1$. For every natural number $i$ the set $R_i(a)$ is nonempty, except for the pairs $(a, i)\in\{(2, 1), (2, 6), (-2, 2), (-2, 3), (3, 1), (-3, 2)\}$.
\end{lem}

For $i\neq 2$ the product of all primitive divisors of $a^i-1$ taken with multiplicities is denoted by $k_i(a)$. Put $k_2(a)=k_1(-a)$. The number $k_i(a)$ is said to be the greatest primitive divisor of $a^i-1$.

It follows from \cite{A1919} that for $i>2$,

\begin{center}

$k_i(a)=\frac{|\Phi_i(a)|}{(r, \Phi_{i_{\{r\}'}}(a))}$
\end{center}
where $\Phi_i(x)$ is the $i$-th cyclotomic polynomial and $r$ is the largest prime dividing $i$; moreover, if $i_{\{r\}'}$ does not divide $r-1$ then $(r,\Phi_{i_{\{r\}'}}(a))=1$.

As a special case of this assertion, we obtain the following lemma.

\begin{lem}\label{kqi}
If $q$ is a power of prime then:
\begin{enumerate}
\item $k_3(q)=(q^2+q+1)/(3,q-1)$

\item $k_4(q)=(q^2+1)/(2,q+1)$

\item $k_6(q)=(q^2-q+1)/(3,q+1)$

\item $k_7(q)=(q^7-1)/(q-1)(7,q-1)$

\item $k_{14}(q)=(q^7+1)/(q+1)(7,q+1)$
\end{enumerate}
\end{lem}

\begin{lem}[{\rm \cite[Lemma 1.4]{GorA2}}]\label{factorKh}
Let $G$ be a finite group, $K\unlhd G$ and put $\overline{G}= G/K$. Take $x\in G$ and $\overline{x}=xK\in G/K$.
The following claims hold

(i) $|x^K|$ and $|\overline{x}^{\overline{G}}|$ divide $|x^G|$.

(ii) If $L$ and $M$ are neigboring members of a composition series of $G$, $L<M$, $S=M/L$, $x\in M$  and
$\widetilde{x}=xL$ is an image of $x$, then $|\widetilde{x}^S|$ divides $|x^G|$.

(iii) If $y\in G, xy=yx$, and $(|x|,|y|)=1$, then $C_G(xy)=C_G(x)\cap C_G(y)$.

(iv) If $(|x|, |K|) = 1$, then $C_{\overline{G}}(\overline{x}) = C_G(x)K/K$.

(v) $\overline{C_G(x)}\leq C_{\overline{G}}(\overline{x})$.
\end{lem}

\begin{lem}[{\rm \cite[Theorem 5.2.3]{Gore}}]\label{Gore5hzOcomutantePstavtomor}
Let $A$ be a $\pi(G)'$-group of automorphisms of
an abelian group $G$. Then $G=C_G(A)\times[G,A]$.
\end{lem}

If $\pi$ is a set of primes, then $n_{\pi}$ denotes the $\pi$-part of $n$, that is, the largest divisor $k$ of $n$ with $\pi(k)\subseteq\pi$; and $n_{\pi'}$ denotes the $\pi'$-part of $n$, that is, the ratio $|n|/n_{\pi}$.


Let $G$ be a group, $a\in G$. We denote by $a^G$ the conjugacy class of the group $G$ containing the element $a$, $Ind_G(a)$ is the index of $C_G(a)$ in the group $G$.
Let $p,q$ be distinct numbers. We say that the group $G$ satisfies the condition $\{p,q\}^* $ (for brevity we will write $G\in\{p, q\}^*$) if for any $\alpha \in N(G) $ we have $\alpha_{\{p, q\}}\in\{|G||_p, |G||_q, |G||_{\{p, q\}}\} $.

\begin{lem}[{\rm \cite[Corollary]{GorBig}}]\label{GorBig}
Let $G\in\{p,q\}^*$ be a group with trivial center, where $p,q\in\pi(G)$ and $p>q>5$. Then $|G|_{\{p,q\}}=|G||_{\{p, q\}} $ and
$C_G(g)\cap C_G(h)=1$ for any $p$- and $q$-elements $g$ and $h$, respectively.
\end{lem}

Define the following function on positive integers:
 \begin{equation*}
\eta(k) =
 \begin{cases}
   k &\text{ if $k$ is odd}\\
   k/2 &\text{if k is even}
 \end{cases}
\end{equation*}

Now we introduce a new function in order to unify further arguments. Namely, given a simple classical group $L$ over a field of order $q$ and a prime $r$ coprime to $q$, we put

\begin{equation*}
\varphi(r,L)=
 \begin{cases}
   e(r,\varepsilon q) &\text{ if $L= A^{\varepsilon}_n(q)$}\\
   \eta(e(r,q))       &\text{if $L$ is symplectic or orthogonal}
 \end{cases}
\end{equation*}

For a classical group $L$, we put $prk(L)$ to denote its dimension if $L$ is a linear or unitary group, and its Lie rank if $L$ is a symplectic or orthogonal group.

\begin{lem}[{\rm\cite[Lemma 2.12]{VasBig}}]\label{vas}
Let $L$ be a simple classical group over a field of order $q$ and characteristic $p$, and let $prk(L)=n\geq4$. If $r\in\pi(L)\setminus\{p\}, i=e(r,q)$, and $n/2<\varphi(r, L)\leq n$, then $L$ includes a cyclic Hall subgroup of order $k_i(q)$.



\end{lem}



\begin{lem}[{\rm\cite[Section 5.1.7]{Carter}}]\label{Regular}
Let $L$ be a simple group of Lie type in characteristic $p$, $L\leq G\leq Aut(L)$. The group $G$ contains an element $h$ such that $(Ind_G (h))_{p'}=|G|_{p'}$.
\end{lem}

\begin{lem}\label{Order}
If $S\leq A\leq Aut(S)$, where $S$ is a non abelian simple group, then $|A|=|A||$.
\end{lem}
\begin{proof}
We show that $|A|_p =|A||_p$ for any $p\in \pi(A)$. To do this, we need to show that there exists $\alpha \in N(A)$ such that $\alpha_p=|A|_p$. Assume that $S$ is a simple alternating group of degree $n$. If $n<100$, then the statement is checked with \cite{GAP}. Suppose that $n\geq100$. Lemma \cite[Lemma 1.11]{GorA2} implies that the interval $\Omega=[n/2..3n/4]$ contains two prime numbers. Hence $\Omega$ contains a prime number $t$ different from $p$. Let $g\in A$ be an element of order $t$. Then $C_A(g)\simeq\langle g \rangle \times R$, where $R$ is an alternating or symmetric group of degree $n-t$. Therefor $R$ contains an element $h$ such that $(Ind_R(h))_p =|R|_p$. Thus, $(Ind_A(gh))_p=|A|_p$.

Assume that $S$ is a group of Lie type over  field in characteristic $t$. Lemma \ref{Regular} implies that $S$ contains a unipotent element $g$ such that $\pi(C_A(g))=\{t\}$. Therefore, if $p\neq t$ then $|Ind_A (g)|_p=|A|_p$. Assume that $p=t$. The set $\pi(S)$ contains a number $r$ that is not adjacent in $GK(S)$ with $p$ (see \cite{VV}). Let $h\in A$ is an element of order $r$. We note that cetraliser of every external automorphism of the group $S$ no contain an element of order $r$. Hence $(Ind_A (h))_p=|A|_p$.

\end{proof}

\begin{lem}\cite[Theorem 1]{vse}\label{pat}
Let $G$ be a finite group, and let $p$ and $q$ be different primes. Then
some Sylow $p$-subgroup of $G$ commutes with some Sylow $q$-subgroup of $G$ if and only
if the class sizes of the $q$-elements of $G$ are not divisible by $p$ and the class sizes of
the $p$-elements of $G$ are not divisible by $q$.
\end{lem}

\begin{lem}\cite{navarro}\label{navarro}
Let $G$ be a finite group and $p$ a prime, $p\not\in\{3, 5\}$. Then $G$ has abelian Sylow $p$-subgroups if and only if $|x^G|_p=1$ for all $p$-elements $x$ of $G$.
\end{lem}

\begin{lem}\label{hz3}
Let $S\leq A\leq Aut(S)$, where $S$ is a simple group of Lie type, $H\in Syl_p(A)$, where $p>5$ is a prime number. If for every $\alpha \in N(A)$ we have $|\alpha|_p \in \{1,|H|\}$, then $|H|$ divides $S$ or $|H|$ divides $A/S$.
\end{lem}
\begin{proof}
Lemma \ref{Order} imply that $|A|=|A||$. In particular $|A||_p=|A|_p$.
Therefor $(Ind_A(x))_p=1$ for any $p$-element $x$. It follows from Lemma \ref{navarro} that the Sylow $p$-subgroups of $A$ is abelian.
Assume that $p$ divides $|S|$ and $|A/S|$. Take a $p$-element $x\in A$ such that $x$ acts on $S$ as an outer automorphism. We have $C_S(x)$ includes a some Sylow $p$-subgroup of $S$. The graph $GK(C_S(x))$ contains a vertex $t$ such that $t$ is not adjacent to $p$ see \cite{VV} and a description of the centralizers of outer automorphisms of groups of Lie type. Hence, there is an element $y\in C_S(x)$ of order $t$ such that $ (Ind_{C_S(x)}(y))_ p=|C_S(x)|_p$. Thus, $|S|_p\leq(Ind_A(xy))_p<|A|_p$; a contradiction.

\end{proof}

\begin{lem}\label{hz5}
Let $S\leq A\leq Aut(S)$, where $S\simeq S(q)$ is a simple group of Lie type over a field of order $q=p^n$ where $p$ is prime, $H$ is a Hall $\pi$-subgroup of $A$, where $\pi\subseteq\pi(A)$ and $\{2,3,5\}\cap\pi=\varnothing$. If for every $\alpha \in N(A)$ we have $|\alpha|_{\pi} \in \{1,|H|\}$, then $|H|$ divides $S$ or $|H|$ divides $A/S$.

\end{lem}
\begin{proof}
It follows from Lemma \ref{hz3} that for any $t\in\pi$, $|A|_t$ divide one of the numbers $|S|$ or $|A/S|$. Assume that there exists numbers $a, b\in\pi$ such that $a$ divides $|S|$ and $b$ divides $|A/S|$. Take a $b$-element $x\in A$. Since $b$ does not divide $|S|$, we see that $x$ acted on $S$ as a fields automorphism of $S$. We have $C_S(x)\simeq S(p^m)$, where $m$ is a non trivial divisor of $n$. The graph $GK(C_S(x))$ contains a vertex $r$ that is not adjacent to $a$ see \cite {VV}. Hence, there exists an element $y\in C_S(x)$ such that $(Ind_{C_S (x)} (y))_t =|C_S (x)|_t$ and $1<(Ind_A (xy))_{\pi}<|H|$; a contradiction.
\end{proof}

\begin{lem}\label{hz2}
Let $L$ be a group of Lie type over a field of order $q$ and characteristic $p$, $H\in Syl_t(L)$, where $t>5$ is prime. If for any $\alpha\in N(L)$ we have $|\alpha|_t\in \{1,|H|\}$ then $H$ is a cyclic or $L\simeq A_1(t^n)$ where $n>1$.
\end{lem}
\begin{proof}
Lemma \ref{Order} implies that $|L|=|L||$. Therefor $(Ind_L(x))_t=1$, for any $t$-element $x$. It follows from \ref{navarro} that a Sylow $t$-subgroup of $L$ is abelian. A Sylow $p$-subgroup of $L$ is abelian iff $L\simeq A_1(q)$. Note that if $t\neq p\in\pi(A_1(q))\setminus\{2\}$, then the Sylow $t$-subgroup of $A_1(q)$ is cyclic. We assume that $L\not\simeq A_1(q)$, and hence $p\neq t$.

Assume that $L$ is a classical group. From the description of the cyclic structure of tors of the finite simple groups see \cite{GrBu}, it follows that $L$ contains an element of order $k_{\varphi(t, L)}(q)$. If $\varphi(t,L)>n/2$, then from Lemma \ref{vas} it follows that the $t$-subgroup of $L$ is cyclic. Assume that $\varphi(t, L)\leq n/2$. Then $|L|_t>|k_{\varphi(t,L)}(q)|_t$. Let $r=|k_{\varphi(t,L)}(q)|_t$, $k=|k_{\varphi(a,L)}(q)|_a$, where $a$ is a prime number such that $\varphi(a, L)=n/2-1$ if $\varphi(t, L)=n/2$ and $\varphi(a,L)= n-\varphi(t, L)$ otherwise.
From the description of the spectra of finite simple groups it follows that $L$ contains an element $h$ of order $r.k$. Therefor $(Ind_L(h^r))_t<|L|_t$, and hence $(Ind_L(h^r))_t=1$. Thus, the centralizer of any $t$-elements contains an element conjugate to $h^r$. Let $T<L$ be a tors having the cyclic structure $\varphi(t, L) + \varphi (t, L)$, $x\in T$ is a $t$-element, $y \in C_L(x)\cap(h^r)^L$. Since $t, a\neq p$, then in $L$ concludes a maximal torus $X$ such that $xy\in X$. Hence $X$ has the cyclic structure $\varphi(t,L) + \varphi(t, L) + \varphi(a, L)>n$. It follows from \cite{GrBu} that $L$ contains no torus with such a cyclic structure; a contradiction.

If $L$ is an exceptional group of Lie type, then the lemma follows from \cite{Der}.

\end{proof}

\begin{lem}\label{hz4}
Let $L$ be a simple group, $L<G\leq Aut(L)$, $H\in Syl_t(L)$, where $t>5$ is prime. If $|H|$ divide $|G/L|$ then $H$ is cyclic.
\end{lem}
\begin{proof}
Since $t>5$, we see that $L$ is group of Lie type. From the fact that $t$ does not divide $|L|$, it follows that $H$ is a subgroup of the group of field automorphisms. The group of field automorphisms of the group $L$ is cyclic and hence $H$ is cyclic.
\end{proof}

\begin{lem}\cite[Corollary 13]{Ah132}\label{Neda} Let $r \in R_{2(n-1)}(q), r_1\in R_n(q), r_2\in R_{n-1}(q)$ and $r_3 \in R_{2(n-2)}(q)$, $L=D_n(q)$, $d=(q^n-1,4)$.

(i) If $x\in L\setminus \{1\}$, then either $(Ind_L(x))_r = |L|_r$ or $(Ind_L(x))_r<|L|_r$
and one of the following holds:

\begin{itemize}

           \item{$Ind_L(x)=|L|d/(|GU_{(n-1)/m}(q^m)|(q+1))$, where $(n-1)/m$
is an odd number and if $q=3$ and $n$ is an even number, then $m\neq 1$;}
           \item{$Ind_L(x) = |L|d/(\alpha|GU_n(q)|)$ , where $n$ is an even number and, if $q=3$, then $\alpha=2$,
if $q\neq3$ and $(q+1)_2>2$, then $\alpha\in{1, 2}$ and otherwise, $\alpha=1$;}
           \item{$(q, n)\neq(3, 2u)$ $(u\in \mathbb{N})$ and $Ind_L(x) = 2|L|d/(|GO^-_{2(n-1)}(q)|(q+1))$;}

            \item{$q\equiv-1(mod 4)$ and $Ind_L(x)=|L|d/(|GO^-_{2(n-1)}(q)|(q+1))$;}

         \end{itemize}

(ii) if $n$ is odd and $x_1\in L\setminus\{1\}$, then $(Ind_L(x_1))_{r_1}=|L|_{r_1}$ or $(Ind_L(x_1))_{r_1}<|L|_{r_1}$ and there exists a divisor $m_1$ of $n$ such that $Ind_L(x_1)=|L|d/(|GL_{n/(m_1)}(q^{m_1})|$.
In particular, if $q\in\{2, 3, 5\}$, then $m_1\neq1$;

(iii) if $n$ is even and $x_2\in L\setminus\{1\}$, then $(Ind_L(x_2))_{r_2}=|L|_{r_2}$ or $(Ind_L(x_2))_{r_2}<|L|_{r_2}$ and one of the following holds:
\begin{itemize}
\item{ $Ind_L(x_2)=|L|d/(|GL_{(n-1)/m_2}(q^{m_2})|(q-1))$, where $m_2$ is a divisor of $n-1$. In particular, if $q\in\{2, 3, 5\}$, then $m_2\neq1$;}
\item{ $q\not\in\{2, 3\}$ and $Ind_L(x_2)=|L|d/(\alpha|GL_n(q)|)$, where if $(q-1)_2 > 2$ and $q\neq5$, then $\alpha\in\{1, 2\}$, if $q=5$, then $\alpha=2$ and otherwise, $\alpha = 1$;}
\item{ $q\not\in\{2, 3, 5\}$ and $Ind_L(x_2)=2|L|d/(|GO^+_{2(n-1)}(q)|(q-1))$;}

\item{ $q\equiv 1(mod 4)$ and $Ind_L(x_2)|=|L|d/(|GO^+_{2(n-1)}(q)|(q-1))$;}
\end{itemize}

(iv) if $n\geq6$ is even and $x_3\in G\setminus\{1\}$, then either $(Ind_L(x_3))_{r_3} = |L|_{r_3}$ or $Ind_L(x_3)|_{r_3}<|L|_{r_3}$ and one of the following holds:
\begin{itemize}
\item{$Ind_L(x_3)=\alpha|L|d/(|GO^{\varepsilon}_{2(n-1)}(q)|(q-\varepsilon1))$, where $\alpha | 2$ under conditions mentioned in $(i,iii)$ and $\varepsilon\in\{+, -\}$;}
    \item{ $Ind_L(x_3)=|L|d/(\alpha\beta)$, for some divisor $\alpha$ of $|\Omega^-_4(q)|$ and $\beta\in\{|GU_{(n-2)/m_3}(q^{m_3})|,|GO^-_{2(n-2)}(q)|,|GO^-_{2(n-2)}(q)|/2\}$, where $(n-2)/m_3$
is an odd number;}
\item{$Ind_L(x_3) = |L|d/q^{2(n-1)}|SP_{2(n-2)}(q)|$;}
\item{$Ind_L(x_3)= |L|d/|GU_{n/2}(q^2)|$, where $n/2$ is an even number.}
\end{itemize}
\end{lem}
\begin{lem}\cite[Lemma 2.6]{Ah11}\label{Lieq3}
Let $G(q)$ be a simple group of Lie type in characteristic $p$, then $|G(q)|<(|G(q)|_p)^3$.
\end{lem}
\begin{lem}\label{Lieq3Aut}
If $G(q)$ be a simple group of Lie type in characteristic $p$, then $|Aut(G(q))|<(|G(q)|_p)^{3,5}$.
\end{lem}
\begin{proof}
The assertion of the lemma follows from the description of the orders of automorphism groups of finite simple groups of Lie type and Lemma \ref{Lieq3}.
\end{proof}

\begin{lem}\label{pchast}
Let $G(q)$ be a simple group of Lie type in characteristic $p$, $t\in \pi(G(q))\setminus\{p\}$. The group $G(q)$ contains an element $h$ such that $(Ind_{G(q)}(h))_p=|G(q)|_p$ and $(Ind_{G(q)}(h))_t<|G(q)|_t$.
\end{lem}
\begin{proof}
It follows from the description of the spectra of finite groups of Lie type see \cite{Bu1, Bu2, Bu3, Bu4} that $G(q)$ contain an element $h$ such that $t\in\pi(|h|)$ and $G(q)$ no contain en element of order $|h|p$. Hence $|C_{G(q)}(h)|$ is not divisible by $p$ and $(Ind_{G(q)}(h))_p=|G(q)|_p$.
\end{proof}

\section{Proof the Theorem for groups $D_4(q)$}

\begin{prop}
Let $L=D_4(q)$, where $q=p^n$. If $G$ is a group with trivial centre and $N(G)=N(L)$ then $G\simeq L$.
\end{prop}

Proposition is proved for $q=2$ and $q=3$ see \cite{Darafsheh}. We will assume that $q>3$.

\begin{lem}\label{good}
$G\in\{r_3(q), r_6(q)\}^*$
\end{lem}
\begin{proof}
Take $h\in L$ such that $(Ind_L(h))_{r_i(q)}<|G|_{r_i(q)}$ where $i\in\{3, 6\}$. Since $r_i(q)$ and $p$ are not adjacent in $GK(L)$, see \cite{VV}, then $h$ is a semisimple element. Hence $h$ lies in some maximal torus $T$ of $L$. It follows from the description of the orders of maximal torus of $L$ see \cite{Car1} that $(Ind_L(h))_{R_i(q)}=1$ and $(Ind_L(h))_{R_{i'}(q)}=|L|_{R_{i}(q)}$, where $i'\in\{3, 6\}\setminus\{i\}$. We have $L\in\{r_3(q), r_6(q)\}^*$. Since $N(G)=N(L)$, we obtain $G\in\{r_3(q), r_6(q)\}^*$.
\end{proof}

\begin{lem}\label{HollD4R3R6}
A Hall $R_3(q)$-subgroup and $R_6(q)$-subgroup of $G$ exists and abelian.
\end{lem}
\begin{proof}
Note that $2,3,5\not\in R_3(q)\cup R_6(q) $. It follows from Lemmas \ref{good} and \ref{GorBig} that for any $r\in R_3(q)\cup R_6(q)$ a Sylow $r$-subgroup of $G$ is abelian. Hence for any $r$-element $h$ we have $(Ind_G(h))_r=1$. It follows from Lemma \ref{Neda} that $(Ind_G (h))_{R_i(q)}=1$, where $i$ such that $r\in R_i(q)$. It follows from Lemma \ref{pat} that the Hall $R_i(q)$-subgroup exists and abelian.

\end{proof}

\begin{lem}\label{factorD4}
The group $G$ includes a normal subgroup $K$ such that $S\leq \overline{G}=G/K\leq Aut(S)$, where $S$ is non abelian simple group, and $|\overline{G}|$ is a multiple of $k_3(q)k_6(q)$.
\end{lem}
\begin{proof}
Let $\widetilde{G}=G/O_{(R_3(q)\cup R_6(q))'} $.
Suppose that $\widetilde{G}$ include non-trivial solvable minimal normal subgroup $X$. We have $\pi(X)\subseteq R_i (q) $, where $i\in \{3, 6\}$. Take a Hall $R_j(q)$-subgroup $H$ of $\widetilde{G}$, where $j\in \{3, 6\}\setminus\{i\}$. Therefor $XH$ is a Frobenius group with kernel $X$. Hence $N(\widetilde{G})$ contains a number is multiple of $|X|$. It follows that $|X|$ divides $k_i(q)$. Since $X$ is a Frobenius kernel of $XH$, we obtain $|X|-1$ is multiple of $|H|$. It follows from $|H|=|G|_{R_j(q)}\geq |L|_{R_j(q)}=k_j(q)$ that $k_j(q)$ divides $k_i(q)$; a contradiction with Lemma \ref{kqi}.

Let $R$ be the socle of $\widetilde {G}$. We get that $R=S_1\times...\times S_k$, where $S_1,...,S_k$ are non abelian simple groups.
Assume that $k> 1$. It follows from the definition that the order of $S_h$ is multiple of some number of $R_3(q)\cup R_6(q)$, for every $1\leq h\leq k$. Note that $\omega(\widetilde{G})\subseteq \omega(G)$. Hence $R$ contains no element of order $tr$, where $t\in R_3(q)$ and $r\in R_6(q)$. Thus, $\pi(R)\cap(R_3(q)\cup R_6(q))\subseteq R_i(q)$, where $i\in\{3,6\}$. The $S_1$ contains an element $g$ such that $(Ind_{R}(g))_{R_3 (q)\cup R_6(q)} =|S_1|_{R_3(q)\cup R_6(q )}$. We have $1<(Ind_{S_1}(g))_{R_3 (q)\cup R_6(q)}<|R|_{R_3(q)\cup R_6(q)}$. Similarly, as in Lemma \ref{Order} we can shows that $|R|=|R||$. In particular $|R|_{R_3(q)\cup R_6(q)}>(Ind_R(g))_{R_3(q)\cup R_6(q)}$; using Lemma \ref {factorKh} we get a contradiction.

Thus $k=1$ and $R\leq \widetilde{G}\leq Aut (R)$, where $R$ is a simple group. Since $|G|_{R_3(q)\cup R_6(q)}=|\widetilde{G}|_{R_3(q)\cup R_6(q )}$ we get $k_3(q)k_6(q)$ divides $|\widetilde{G}|$.

\end{proof}

Let $S$, $\overline{G}$ and $K$ be similar as in Lemma \ref{factorD4}.

\begin{lem}\label{AutD4}
$|Out(S)|$ is not a multiple of $k_3(q)$ and $k_6(q)$.
\end{lem}
\begin{proof}
Assume that there exists $i\in\{3,6\}$ such that $k_i(q)$ divides $|Out(S)|$.
Since $r_i(q)>5$, we get $S$ can not be isomorphic to a sporadic group or an alternating group. Hence $S$ is a group of Lie type over a field of order $l=t^r$ of characteristic $t$. From Lemmas \ref{Order} and \ref{factorKh} it follows that $|S|$ divides $|G||$. Since $|G||=|L||$, we obtain $|S|$ divides $|L|$. In particular, for any $k\in \pi (S) $, we have $|S|_k\leq|L|_k$.

Suppose that $S$ is a group of exceptional Lie type. From $2,3\not\in R_i (q) $, it follows that $k_i(q)$ divides order of the field automorphisms group of $S$. Hence $l\geq t^{k_i (q)}$. It follows from Lemma \ref{kqi} that $k_i(q)\in\{(q^2-q+1)/(3, q+1), (q^2+q+1)/(3, q-1)$. Note that for any $k\in\pi(L) $, we have $|L|_k<|L|_p$ and $|L|_p=q^{12}$.
If $S\not\simeq\ ^2G_2(l)$ then $|S|_t> l^4 $. Since $q>3$, we have $|S|_t\geq t^{4k_i (q)}>q^{12}=|L|_p$; a contradiction. If $S\simeq\ ^2G_2(l)$, and $l>2$ or $q>4$, then we obtain $|S|_t>|L|_p$. Assume that $l= 2, q=4$. Then $43 \in\pi(S)\setminus\pi(L)$; a contradiction.

Assume that $S$ is a classical group of Lie type and $k_i(q)$ divides order of the fields automorphism group. If $|S|_t>l^3$, then similarly as above we get a contradiction. Assume that $|S|_t=l^3$, then $S$ is isomorphic to one of the groups $A_2(l)$ or$\ ^2A_2(l)$. If $t>2$ or $q>4$, then $|S|_t>|L|_p$; a contradiction. We get $t=2, q=4$. In this case $k_3(4)=13, k_6(4)=7$. If $l \geq t^{13}$, then we obtain $|S|_t>|L|_p$; a contradiction. Hence $l=t^7$ and $127\in \pi(S) \setminus \pi (L)$; a contradiction.

Assume that $|S|_t=l^2$. Therefor $S$ is isomorphic to$\ ^2B_2(r)$. If $q>5$ or $t>3 $, then $|S|_t>|L|_p$; a contradiction. Assume that $t=3$ and $|S|_t\leq |L|_p $. Then $1093 \in\pi(S)\setminus \pi (L) $; a contradiction. Assume that $t=2$, then $127\in \pi(S)\setminus\pi(L)$; a contradiction.

Assume that $|S|_t=l$. Hence $S$ is isomorphic to $A_1(r)$. From the fact that $|S|_t\leq|L|_p$ it follows that $q<11$. Analyzing all possible variants similarly as above we arrive at a contradiction.

Thus, there exists $r_i(q)$, such that $r_i(q)$ divides order of the diagonal automorphisms group of $S$. In this case, $S$ is isomorphic to one of the groups $A^{\varepsilon}_m(l)$ where $\varepsilon \in \{+, -\} $. We have $r_i(q)$ divides $(n+1, l-(\varepsilon1))$. It is easy to shows that in this case $|S|_t\geq|L|_p$; a contradiction.

\end{proof}

\begin{lem}\label{Aut2D4}
The number $|\overline{G}|/|S|$ is not a multiple of $r_i(q)$, where $i\in\{3,6\}$.
\end{lem}
\begin{proof}
Assume that there exists $r_i(q)\in\pi(|\overline{G}|/|S|)$, where $i\in\{3,6\}$. Take a Hall $R_i(q)$-subgroup $H$ of $\overline{G}$. We have for every $h\in\overline{G}$, $(Ind_{\overline{G}}(h))_{R_i(q)}=1$ or $(Ind_{\overline{G}}( h))_{R_i(q)}=k_i(q)$. Using Lemma \ref{hz5}, we obtain $|H|$ divides $|S|$ or $|\overline{G}|/|S|$. Since $r_i(q)$ divides $|\overline{G}|/|S|$, we see that $k_i(q)$ divides $|\overline{G}|/|S|$; a contradiction with Lemma \ref{AutD4}.
\end{proof}

\begin{lem}\label{R4}
$R_4(q)\cap\pi(K)=\varnothing$.
\end{lem}
\begin{proof}
Assume that $R_4(q)\cap\pi(K)\neq\varnothing$. Let $\widetilde{\ }:G\rightarrow G/O_{R_4 (q)'}(K)$ be a natural homomorphism, $R<\widetilde{G}$ be a minimal normal subgroup. Take $h\in G$ such that $\pi(|h|)\subseteq R_6(q)$. It follows from Lemma \ref{Aut2D4} that $\overline{h}\in S$, where $\overline{h}\in\overline{G}$ is the image of the element $h$. By the definition we have $|\pi(R)\cap R_4 (q)|\neq\varnothing$. If $R$ is not solvable, then it is easy to show that $R$ contains an $R_4(q)$-element $x$ such that $(Ind_{\widetilde{G}}(x))_{R_3(q)}<|G|_{R_3(q)}$ and $(Ind_{\widetilde{G}}(x))_{R_6(q)}<|G|_{R_6(q)}$. We have $\pi(|x|)\cap\pi(O_{R_4 (q)'}(K))=\varnothing$. It follows from Lemma \ref{factorKh} that $G$ contains $y$ such that $(Ind_G(y))_{R_3 (q)\cup R_6 (q)}=1$; a contradiction with the fact that $G\in\{r_3(q), r_6(q)\}^*$ and Lemma \ref{GorBig}. Hence, $R$ is an elementary Abelian $r$-group, where $r\in R_4(q)$. Assume that $C_R(h)>1$. Let $x\in C_R(h)$. Then $(Ind_{\widetilde{G}}(x))_{R_6 (q)}=1$. It follows from Lemma \ref{Neda} that for any $\alpha\in N(G)$ such that $\alpha_{R_6(q)}=1$, we have $\alpha_{R_4 (q)}=|L|_{R_4(q)}$. Therefor $(Ind_{\widetilde{G}}(hx))_{R_4(q)}=(Ind_{\widetilde{G}}(x))_{R_4(q)}$. Thus $C_{\widetilde{G}}(h)$ contains a some Sylow $r$-subgroup of $C_{\widetilde {G}}(x)$. In particular, $h$ acts trivially on $R$. Since $R$ is normal subgroup of $\widetilde {G} $, we obtain any element conjugate with $h$ acts trivially on $R$. In particular, the minimal preimage $H<\widetilde{G}$ of $S$ acts trivially on $R$. If $|S|_{R_3(q)}>1$, then $G$ contains an element $y$ such that $Ind_G(y)_{R_3(q)\cup R_6(q)}=1$; a contradiction. Thus, $|\overline{G}/S|$ is a multiple of $k_3(q) $; a contradiction with Lemma \ref{AutD4}.

Therefor $C_{R}(h)=1$. Thus, $|R|$ divides $Ind_{\widetilde{G}}(h)$. Since $(Ind_{\widetilde {G}}(h))_r\leq|L|_{r}$, we obtain $|R|\leq(k_4(q))^2$. The group $H$ acts faithfully on $R$. Hence $|R|-1\geq|H|$. From Lemma \ref{Aut2D4} we get that $|H|\geq 4k_3(q)k_6(q)>(k_4(q))^2$; a contradiction.

\end{proof}

\begin{lem}\label{D4Alt}
The group $S$ is not isomorphic to an alternating group.
\end{lem}
\begin{proof}
Assume that $S\simeq Alt_m$ for some $m>5$. From Lemmas \ref{factorD4} and \ref{R4} it follows that $|\overline{G}|$ is a multiple of $(k_4 (q))^2k_3(q)k_6(q)$. Since $|Out(S)|\leq4$ and $|S|$ is a multiple to $3$ and $4$, we obtain $|S|$ is divisible by $12(k_4(q))^2k_3(q)k_6(q)=(q^2+1)^2(q^4+q^2+1)$.
From Lemma \ref{HollD4R3R6}, we have that a Halls $R_3(q)$-subgroup and $R_6(q)$-subgroup of $S$ are abelian. Lemmas \cite{Hall} and \cite{Tom66} implies that a Hall $\pi$-subgroup of $Alt_m$ is abelian only if $|\pi\cap\pi(Alt_m)|=1$. Hence $|R_3(q)|=|R_6(q)|=1$, and $m\geq max((q^2+q+1)/(3, q-1), (q^2-q+1)/(3, q+1))$. If $q\geq 5$, then $|S|>|L| $. Hence $S$ contains an element $h$ such that $Ind_S(h)$ does not divide any number from $N(L)$; a contradiction. If $q=4$ then $R_4(q)=\{17\}$. Hence $m>34$ and $|S|>|L|$; a contradiction.

\end{proof}

\begin{lem}\label{Sporadic}
The group $S$ is not isomorphic to any of the sporadic groups.
\end{lem}
\begin{proof}
The assertion of the lemma follows from the fact that $(q^2+1)^2(q^4+q^2+1)$ divides $|\overline{G}|$, $|\overline{G}|$ divides $|L|$ and \cite{Atlas}.
\end{proof}
\begin{lem}\label{R1R2}
$(R_1(q)\cup R_2(q))\cap\pi(K)=\varnothing$.
\end{lem}
\begin{proof}
Assume that there exists $r\in(R_1(q)\cup R_2(q))\cap\pi(K)$. Let $R=O_{r'}(G)$, $X\unlhd K/R $ be a minimal normal subgroup. Assume that $S\simeq A_1(t^m)$, where $t\in R_3(q)\cup R_6(q)$. We have $|\overline{G}|\leq2mt^m(t^m-1)^2$, $|S|_t=k_j(q)$ where $j\in\{3,6\}$. It follows from Lemmas \ref{factorD4} and \ref{R4} that $|\overline{G}|$ is divisible by $(k_4(q))^2k_3(q)k_6 (q)$; a contradiction. Hence $S\not\simeq A_1(t^m)$ where $t\in R_3(q)\cup R_6(q)$.

It follows from Lemmas \ref{hz2} and \ref{hz4} that for any $t\in R_3(q)\cup R_6(q)$ a Sylow $t$-subgroup of $\overline{G}$ is cyclic. It follows from Lemma \ref{HollD4R3R6} that a Hall $R_3(q)$-subgroup and $R_6(q)$-subgroup of $\overline{G}$ are cyclic. Let $i=3$ if $r\in R_1(q)$ and $i=6$ if $r\in R_2(q)$, $H\in Hall_{R_i(q)}(G/R)$. It follows from Lemma \ref{Neda} that if $r\neq 2$ then for any $x\in H $ we have $(Ind_{G/R}(x))_{r}\leq 1/4(q +(-1)^e)^2$, and $(Ind_G(x))_{r}\leq 4(q+(-1)^e)^2$ if $r=2$, where $e\in\{1,2\}$. Let $h$ be the generating element of $H$, and let $x\in H$ be an element of order $t\in R_i(q)$. From Lemma \ref{Gore5hzOcomutantePstavtomor} it follows that $X=C_X(H)\times[X, H]$. Hence $(Ind_X(h))_r\geq|[X, H]|_r$. For any $x\in H$, we have that $X\neq C_X(x)$. Therefor $|[X, H]|>|H|$. If $r\neq 2$ then $|Ind_X(h)|_{r}|\geq|[X,H]|>|Ind_{G/R}(h)|_{r}$; a contradiction. If $r=2$, then $(Ind_S (\overline{h}))_2\geq 4$, hence $(Ind_{G/R}(h))_2\geq 4|[X, H]|$; a contradiction.
\end{proof}

From Lemmas \ref{factorD4}, \ref{R4} and \ref{R1R2} it follows that $\pi(K)\in\{p\}$.

\begin{lem}\label{LieType}
$S$ is a group of Lie type over the field of characteristic $p$.
\end{lem}
\begin{proof}
From Lemmas \ref{D4Alt} and \ref{Sporadic} it follows that $S$ is a group of Lie type. Let $t$ be the characteristic of the field over which $S$ is given. From Lemma \ref {Lieq3}, it follows that $|S|<|S|^3_t$. It follows from Lemmas \ref{AutD4}, \ref{R4} and \ref{R1R2} that $|\overline{G}|$ is divisible by $(k_1(q)k_2(q))^ 4(k_4(q))^2k_3(q)k_6(q)$. Lemma \ref{Lieq3Aut} implies that $t\not\in R_3(q)\cup R_6(q)$. Assume that $t\in R_1(q) \cup R_2(q)$. If $t\in R_1(q)$ put $r\in R_3(q)\cup\pi(S)$. If $t\in R_2(q)$ put $r\in R_6(q)\cup\pi(S)$. From Lemma \ref{pchast}  it follows that $S$ contains an element $h$ such that $(Ind_S(h))_t=|S|_t$ and $(Ind_S(h))_r<|S|_r$. Since $G\in \{r_3(q), r_6(q)\}^*$, we obtain $(Ind_S(h))_r=1$. It follows from Lemma \ref{Neda} that $(Ind_S(h))_t\leq 4(k_i(q))^2_t$, where $i$ is such that $t\in R_i(q)$; a contradiction with Lemma \ref{Lieq3Aut}.

 Assume that $t\in R_4(q)$. From Lemma \ref{Regular} it follows that $S$ contains an element $h$ such that $(Ind_{\overline{G}}(h))_{t'}=|\overline{G}|_ {t'}$. Using Lemma \ref {Neda} we obtain $\alpha_2<|\overline{G}|_2$ for any $\alpha\in N(\overline{G})$ such that $\alpha_{r_4(q)}<|\overline{G}|_{r_4(q)}$; a contradiction.
Thus, $t=p$.

\end{proof}

\begin{lem}\label{Liner}
$S\simeq D_4(q)$.
\end{lem}
\begin{proof}
It follows from Lemmas \ref{factorD4}, \ref{R4}, and \ref{R1R2} that $|\overline{G}|$ is a multiple of $|L|_{p'}$, in particular $|L|_p\geq |S|_p$. Lemma \ref{LieType} implies that $S$ is a group of Lie type over the field of characteristic $p$. From Lemmas \ref{hz2} and \ref{hz4}, it follows that any Sylow $t$-subgroup of $\overline{G}$ is cyclic, where $t\in R_3(q)\cup R_6(q)$. It follows from Lemma \ref{HollD4R3R6} that the Hall $R_3(q)$-subgroup and $R_6(q)$-subgroup of $\overline{G}$ are cyclic. It follows from Lemma \ref{AutD4} that $k_3(q)k_6(q)$ divides $|S|$. Hence $k_i(q)$ divides the order of some maximal torus, where $i\in\{3,6\}$. Let $a$ be a maximal number such that the set $\pi(S)\cap R_{a}(p)$ is not empty. From Lemma \ref{factorD4} and the fact that $\pi(S)\leq \pi(L)$ it follows that $a=6n$. We have $|S|_p\leq|L|_p=p^{2a}$.

Assume that $S\simeq A_m(p^l)$. The maximal number $b$ such that $R_b(p)\subseteq\pi(S)$ is equal to $l(m+1)$. We have $b=a=l(m+1)$, $|S|_p=p^{lm(m+1)/2}=p^{am/2}\leq p^{2a}$. Thus, $m\leq4$ and $l=6n/m$. From the fact that $|S|_{p'}\leq |L|_{p'}\leq|Aut(S)|_{p'}$, it is easy to get a contradiction.

Assume that $S\simeq\ ^2A_m(p^l)$. We have $a=2lm$ if $m$ is even and $a=2l(m+1)$ if $m$ is odd, $|S|_p=p^{lm(m+1)/2}$. Thus, $m\leq7$ and $l=3n/m$ if $m$ is even and $l=3n/(m+1)$ if $m$ is odd. From the fact that $|S|_{p'}\leq|L|_{p'}\leq|Aut(S)|_{p'}$, it is easy to get a contradiction.

Assume that $S\simeq B_m(p^l)$ or $S\simeq C_m(p^l)$. We have $a=2lm$, $|S|_p=p^{lm^2}=p^{am/2}\leq p^{2a}$. Thus, $m\leq4$ and $l=3n/m$. From the fact that $ |S|_{p'}\leq|L|_{p'}\leq|Aut(S)|_{p'}$, it is easy to get a contradiction.

Assume that $S\simeq\ ^2B_2(p^l)$. We have $a=4l$, $|S|_p=p^{2l}$. Hence $l=3n/2$. Thus, $|Aut(S)|_{p'}<q(q-1)(q^3+1)$; a contradiction with the fact that $q^{12}<|L|_{p'}\leq|Aut(S)|_{p'}$.

Similarly show that $S$ can not be isomorphic to$\ ^3D_4(p^l), G_2(p^l),\ ^2G_2(p^l), F_4(p^l),\ ^2F_4(p^l)$.

Assume that $S\simeq D_m(p^l)$, where $m\geq4$. We have $a=2l(m-1)$, $|S|_p=p^{lm(m-1)}=p^{am/2}\leq p^{2a}$. Thus, $m\leq4$. Hence, $l=n$ and $S\simeq L$.

Assume that $S\simeq\ ^2D_m(p^l)$. We have $a=2lm$, $|S|_p=p^{lm(m-1)}=p^{a(m-1)/2} \leq p^{2a}$. Thus, $m\leq5$. From the fact that $|S|_{p'}\leq |L|_{p'}\leq|Aut(S)|_{p'}$, it is easy to get a contradiction.

Assume that $S\simeq\ ^{\varepsilon}E_{\alpha}(p^l)$, where $\alpha \in \{6,7,8\}$, if $\alpha=6$ then $\varepsilon \in \{+,-\}$ otherwise $\alpha$ is the empty symbol. In this case it is easy to show that $|S|_p>|L|_p$; a contradiction.

\end{proof}

\begin{lem}
$G\simeq L$
\end{lem}
\begin{proof}
From Lemma \ref{Liner} it follows that $S\simeq L$. Assume that $K$ is not trivial. From Lemmas \ref {factorD4}, \ref {R4}, and \ref{R1R2} it follows that $K$ is a $p$-group. Take $h\in S$ is an element of order $k_6(q)$. Since $r_6(q)$ is not adjacent to $p$ in the graph $GK(S)$, we obtain $(Ind_S(h))_p=|S|_p=|L|_p$. Let $h'\in G$ be the preimage of the element $h$. If $h'$ acts non trivially on $K$ then $(Ind_G(h'))_p>|L|_p$; a contradiction. Thus $h'$ acts trivially on $K$. Hence any element conjugate to $h'$ acts trivially on $K$. Since $S$ is a simple group, the minimal preimage $H$ of $ S $ is contained in $C_G(K)$. Hence for any $x\in K$ we have $(Ind_G(x))_{R_3(q)\cup R_6(q)}=1$; a contradiction with Lemma \ref {GorBig}. Thus $K$ is trivial.

From Lemma \ref{Order} we obtain $S=G$.
\end{proof}

The proposition $1$ is proved.

\section{Proof the Theorem for groups $D_8(q)$}

\begin{prop}
Let $L=D_8(q)$, where $q=p^n$. If $G$ is a group with trivial centre and $N(G)=N(L)$ then $G\simeq L$.
\end{prop}

\begin{lem}\label{good2}
$G\in\{r_7(q), r_{14}(q)\}^*$
\end{lem}
\begin{proof}
The assertion of the lemma follows from Lemma \ref{Neda}.
\end{proof}
\begin{lem}\label{HollD8R3R6}
A Hall $R_7(q)$-subgroup and $R_{14}(q)$-subgroup of $G$ exists and is abelian.
\end{lem}
\begin{proof}
Note that $2,3,5\not\in R_7(q)\cup R_{14}(q)$. It follows from Lemmas \ref{good2} and \ref{GorBig} that for any $r\in R_7(q)\cup R_{14}(q)$ a Sylow $r$-subgroup of $G$ is abelian. Hence for any $r$-element $h$ we have $(Ind_G(h))_r=1$. It follows from Lemma \ref{Neda} that $(Ind_G (h))_{R_i(q)}=1$, where $i$ such that $r\in R_i(q)$. It follows from Lemma \ref{pat} that the Hall $R_i(q)$-subgroup exists and is abelian.
\end{proof}

\begin{lem}\label{factorD8}
The group $G$ includes a normal subgroup $K$ such that $S\leq \overline{G}=G/K\leq Aut(S)$, where $S$ is non abelian simple group, and $|\overline{G}|$ is a multiple of $k_7(q)k_{14}(q)$.
\end{lem}
\begin{proof}
Let $\widetilde{G}=G/O_{(R_7(q)\cup R_{14}(q))'} $.
Suppose that $\widetilde{G}$ include non-trivial solvable minimal normal subgroup $X$. We have $\pi(X)\subseteq R_i(q)$, where $i\in \{7, 14\}$. Take a Hall $R_j(q)$-subgroup $H$ of $\widetilde{G}$, where $j\in \{7, 14\}\setminus\{i\}$. Therefor $XH$ is a Frobenius group with kernel $X$. Hence $N(\widetilde{G})$ contains a number is multiple of $|X|$. It follows that $|X|$ divides $k_i(q)$. Since $X$ is a Frobenius kernel of $XH$, we obtain $|X|-1$ is multiple of $|H|$. It follows from $|H|=|G|_{R_j(q)}\geq |L|_{R_j(q)}=k_j(q)$ that $k_j(q)$ divides $k_i(q)$; a contradiction with Lemma \ref{kqi}.

Let $R$ be the socle of $\widetilde {G}$. We get that $R=S_1\times...\times S_k$, where $S_1,...,S_k$ are non abelian simple groups.
Assume that $k> 1$. It follows from the definition that the order of $S_h$ is multiple of some number of $R_7(q)\cup R_{14}(q)$, for every $1\leq h\leq k$. Note that $\omega(\widetilde{G})\subseteq \omega(G)$. Hence $R$ contains no element of order $tr$, where $t\in R_7(q)$ and $r\in R_{14}(q)$. Thus, $\pi(R)\cap(R_7(q)\cup R_{14}(q))\subseteq R_i(q)$, where $i\in\{7, 14\}$. The $S_1$ contains an element $g$ such that $(Ind_{R}(g))_{R_7 (q)\cup R_{14}(q)} =|S_1|_{R_{7}(q)\cup R_{14}(q )}$. We have $1<(Ind_{S_1}(g))_{R_7 (q)\cup R_{14}(q)}<|R|_{R_7(q)\cup R_{14}(q)}$. Similarly, as in Lemma \ref{Order} we can shows that $|R|=|R||$. In particular $|R|_{R_7(q)\cup R_{14}(q)}>(Ind_R(g))_{R_7(q)\cup R_{14}(q)}$; using Lemma \ref{factorKh} we get a contradiction.

Thus $k=1$ and $R\leq \widetilde{G}\leq Aut (R)$, where $R$ is a simple group. Since $|G|_{R_7(q)\cup R_{14}(q)}=|\widetilde{G}|_{R_7(q)\cup R_{14}(q )}$ we get $k_7(q)k_{14}(q)$ divides $|\widetilde{G}$.
\end{proof}

Let $S$, $\overline{G}$ and $K$ be similar as in Lemma \ref{factorD8}.

\begin{lem}\label{AutD8}
$|Out(S)|$ is not a multiple of $k_7(q)$ and $k_{14}(q)$.
\end{lem}
\begin{proof}
Assume that there exists $i\in\{7, 14\}$ such that $k_i(q)$ divides $|Out(S)|$.
Since $r_i(q)>13$, we get $S$ can not be isomorphic to a sporadic group or an alternating group. Hence $S$ is a group of Lie type over a field of order $l=t^r$ of characteristic $t$. From Lemmas \ref{Order} and \ref{factorKh} it follows that $|S|$ divides $|G||$. Since $|G||=|L||$, we obtain $|S|$ divides $|L|$. In particular, for any $k\in \pi (S) $, we have $|S|_k\leq|L|_k$.

From $2,3\not\in R_i (q) $, it follows that $k_i(q)$ divides order of the field automorphisms group of $S$. Hence $l\geq t^{k_i (q)}$. It follows from Lemma \ref{kqi} that $k_i(q)\in\{(q^7-1)/((7, q-1)(q-1)), (q^7+1)/((7, q+1)(q+1)\}$. Note that for any $k\in\pi(L)$, we have $|L|_k<|L|_p$ and $|L|_p=q^{56}$.
Since $q>3$, we have $|S|_t\geq t^{k_i (q)}>q^{56}=|L|_p$; a contradiction.
\end{proof}

\begin{lem}\label{Aut2D8}
The number $|\overline{G}|/|S|$ is not a multiple of $r_i(q)$, where $i\in\{7,14\}$.
\end{lem}
\begin{proof}
Assume that there exists $r_i(q)\in\pi(|\overline{G}|/|S|)$, where $i\in\{7,14\}$. Take a Hall $R_i(q)$-subgroup $H$ of $\overline{G}$. We have for every $h\in\overline{G}$, $(Ind_{\overline{G}}(h))_{R_i(q)}=1$ or $(Ind_{\overline{G}}( h))_{R_i(q)}=k_i(q)$. Using Lemma \ref{hz5}, we obtain $|H|$ divides $|S|$ or $|\overline{G}|/|S|$. Since $r_i(q)$ divides $|\overline{G}|/|S|$, we see that $k_i(q)$ divides $|\overline{G}|/|S|$; a contradiction with Lemma \ref{AutD8}.
\end{proof}

\begin{lem}\label{R8}
$R_8(q)\cap\pi(K)=\varnothing$.
\end{lem}
\begin{proof}
Assume that $R_8(q)\cap\pi(K)\neq\varnothing$. Let $\widetilde{\ }:G\rightarrow G/O_{R_8(q)'}(K)$ be a natural homomorphism, $R<\widetilde{G}$ be a minimal normal subgroup. Take $h\in G$ such that $\pi(|h|)\subseteq R_{14}(q)$. It follows from Lemma \ref{AutD8} that $\overline{h}\in S$, where $\overline{h}\in\overline{G}$ is the image of the element $h$. By the definition we have $|\pi(R)\cap R_8 (q)|\neq\varnothing$. If $R$ is not solvable, then it is easy to show that $R$ contains an $R_8(q)$-element $x$ such that $(Ind_{\widetilde{G}}(x))_{R_7(q)}<|G|_{R_7(q)}$ and $(Ind_{\widetilde{G}}(x))_{R_{14}(q)}<|G|_{R_{14}(q)}$. We have $\pi(|x|)\cap\pi(O_{R_8 (q)'}(K))=\varnothing$. It follows from Lemma \ref{factorKh} that $G$ contains $y$ such that $(Ind_G(y))_{R_7 (q)\cup R_{14}(q)}=1$; a contradiction with the fact that $G\in\{r_7(q), r_{14}(q)\}^*$ and Lemma \ref{GorBig}. Hence, $R$ is an elementary Abelian $r$-group, where $r\in R_8(q)$. Assume that $C_R(h)>1$. Let $x\in C_R(h)$. Then $(Ind_{\widetilde{G}}(x))_{R_{14}(q)}=1$. It follows from Lemma \ref{Neda} that for any $\alpha\in N(G)$ such that $\alpha_{R_{14}(q)}=1$, we have $\alpha_{R_8 (q)}=|L|_{R_8(q)}$. Therefor $(Ind_{\widetilde{G}}(hx))_{R_8(q)}=(Ind_{\widetilde{G}}(x))_{R_8(q)}$. Thus $C_{\widetilde{G}}(h)$ contains a some Sylow $r$-subgroup of $C_{\widetilde{G}}(x)$. In particular, $h$ acts trivially on $R$. Since $R$ is normal subgroup of $\widetilde{G} $, we obtain any element conjugate with $h$ acts trivially on $R$. In particular, the minimal preimage $H<\widetilde{G}$ of $S$ acts trivially on $R$. If $|S|_{R_7(q)}>1$, then $G$ contains an element $y$ such that $Ind_G(y)_{R_7(q)\cup R_{14}(q)}=1$; a contradiction. Thus, $|\overline{G}/S|$ is a multiple of $k_7(q) $; a contradiction with Lemma \ref{AutD8}.

Therefor $C_{R}(h)=1$. Thus, $|R|$ divides $Ind_{\widetilde{G}}(h)$. Since $(Ind_{\widetilde {G}}(h))_r\leq|L|_{r}$, we obtain $|R|\leq(k_8(q))^2$. The group $H$ acts faithfully on $R$. Hence $|R|-1\geq|H|$. From Lemma \ref{Aut2D8} we get that $|H|\geq 4k_7(q)k_{14}(q)>(k_8(q))^2$; a contradiction.
\end{proof}

\begin{lem}\label{D8Alt}
The group $S$ is not isomorphic to an alternating group.
\end{lem}
\begin{proof}
Assume that $S\simeq Alt_m$ for some $m>5$. Since $S$ is divided by $r_{14}(q)>13$, we get $m>13$. Therefor $|Out(S)|=2$.
From Lemma \ref {HollD8R3R6}, it follows that the Hall $R_7(q)$- and $R_{14}(q)$-subgroups are abelian. It follows from Lemmas \cite{Hall} and \cite{Tom66} that the Hall $\pi$-subgroup of $Alt_m$ is abelian only if $|\pi\cap\pi(Alt_m)|=1$. Thus, $m\geq k_i(q)$, where $i\in\{7, 14\} $. Hence $|S|>|L|$; a contradiction.
\end{proof}

\begin{lem}\label{Sporadic2}
The group $S$ is not isomorphic to any of the sporadic groups.
\end{lem}
\begin{proof}
The assertion of the lemma follows from the fact that $k^2_8(q)k_7(q)k_{14}(q)$ divides $|\overline{G}|$, $|\overline{G}|$ divides $|L|$ and \cite{Atlas}.
\end{proof}

\begin{lem}\label{MnogoR}
$\pi(K)\subseteq\{p,2\}\cup R_4(q)$.
\end{lem}
\begin{proof}
Assume that $r_j(q)\in\pi(K)$, where $j\neq 4$ and $r_j(q)\neq2$. From Lemmas \ref{factorD8} and \ref{R8} it follows that $j\not\in\{7,8,14\}$. Let $R=O_{r_j'}(G)$, $X\unlhd K/R$ be the minimal normal subgroup. Similarly, as in the Lemma \ref{AutD4}, we can show that $X$ is an elementary abelian group. Let $i=14$ if $j=2$ in the another case $i=7$. It follows from Lemmas \ref{hz2} and \ref{hz4} that every $s$-subgroup of $\overline{G}$ is cyclic, where $t\in R_7(q) \cup R_{14}(q)$. It follows from Lemma \ref{HollD8R3R6} that the Hall $R_7(q)$- and $R_{14}(q)$-subgroups of $\overline{G}$ are cyclic. Let $H\in Hall_{R_i(q)}(G/R)$, and $h\in H$ be the generating element. Then $h$ acts regularly on $[X,H]$. Hence $(Ind_{G/R}(h))_{r_j(q)}>|[X,H]|\geq k_i(q)+ 1$. if $j\in\{3,6,5,10,12\}$ then $k_i(q)>|L|_{r_l(q)}$; a contradiction.
If $j\in\{1, 2\}$ and $r_j(q)\neq\{2 \}$, then $(Ind_{G/R}(h))_{r_j(q)}\leq(k_j(q))^6<k_i(q) $; a contradiction.
\end{proof}

\begin{lem}\label{LieType8}
The $S$ is a group of Lie type over the field of characteristic $p$.
\end{lem}
\begin{proof}
From Lemmas \ref{D8Alt} and \ref{Sporadic2} it follows that $S$ is a group of Lie type over a field of order $t^m$ for some prime $t$. The assertion of the lemma follows from the fact that $|\overline{G}|<|S|^3_tt^m$ and Lemmas \ref{factorD8}, \ref{R8}, \ref{MnogoR}.
\end{proof}
\begin{lem}\label{Liner8}
$S\simeq D_8(q)$.
\end{lem}
\begin{proof}
It follows from Lemmas \ref{factorD8}, \ref{R8}, and \ref{MnogoR} that $|\overline{G}|$ is a multiple of $|L|_{(\{p,2\}\cup R_4(q))'}$. Lemma \ref{LieType8} implies that $S$ is a group of Lie type over the field of characteristic $p$. From Lemmas \ref{hz2} and \ref{hz4}, it follows that any Sylow $t$-subgroup of $\overline{G}$ is cyclic, where $t\in R_7(q)\cup R_{14}(q)$. It follows from Lemma \ref{HollD8R3R6} that the Hall $R_7(q)$-subgroup and $R_{14}(q)$-subgroup of $\overline{G}$ are cyclic. It follows from Lemma \ref{Aut2D8} that $k_7(q)k_{14}(q)$ divides $|S|$. Hence $k_i(q)$ divides the order of some maximal torus, where $i\in\{7, 14\}$. Let $a$ be a maximal number such that the set $\pi(S)\cap R_{a}(p)$ is not empty. From Lemma \ref{factorD4} and the fact that $\pi(S)\leq \pi(L)$ it follows that $a=14n$. We have $|S|_p\leq|L|_p=p^{4a}$.

Assume that $S\simeq A_m(p^l)$. We have $a=l(m+1)$, $|S|_p=p^{lm(m+1)/2}=p^{am/2}\leq p^{4a}$. Thus, $m\leq8$ and $l=(14n-1)/m$. Thus $S$ contains a torus of order $(q^{14}-1)/(m+1, p^l-1)$. In particular $S$ contains an element of order $r_{n7}(p)r_{n14}(p)$; a contradiction with Lemma \ref{good2}.

Assume that $S\simeq\ ^2A_m(p^l)$. We have $a=2l(m+1)$ if $m$ is even and $a=2lm$ if $m$ is odd, $|S|_p=p^{lm(m+1)/2}$. Thus, $m\leq8$ and $l=3n/m$ if $m$ is even and $l=3n/(m + 1) $ if $ m $ is odd. It is easy to obtain a contradiction from the fact that $|Aut(S)|_{p'}>|L|_{(\{p,2\}\cup R_4(q))'}$.

Assume that $S\simeq B_m(p^l)$ or $S\simeq B_m(p^l)$. We have $a=2lm$, $|S|_p=p^{lm^2}=p^{am/2}\leq p^{2a}$. Thus, $m\leq4$ and $l=3n/m$. From the fact that $|S|_{p'}\leq |L|_{(\{2, p\}\cup R_4(q))'}\leq|Aut(S)|_{p'}$, it is easy to obtain a contradiction.

Assume that $S\simeq\ ^2B_2(p^l)$. We have $l\leq a<4l$, $|S|_p=p^{2l}$. Hence $14n\leq l<7n/2$. Thus, $|Aut(S)|_{p'}<q(q-1)(q^2+1)$; a contradiction with the fact that $q^{12}<|L|_{(\{2,p\}\cup R_4(q))'}\leq|Aut(S)|_{p'}$.

Similarly show that $S$ can not be isomorphic to$\ ^3D_4(p^l), G_2(p^l),\ ^2G_2(p^l), F_4(p^l),\ ^2F_4(p^l)$.

Assume that $S\simeq D_m(p^l)$. We have $a=2l(m-1)$, $|S|_p=p^{lm(m-1)}=p^{am/2}\leq p^{4a}$. Thus, $m\leq8$. If $m=8$ then $l=n$ and consequently $S\simeq L$. Assume that $m<8$. If $m=7$ or $m=5$ then from Lemma \ref{Neda} we yet that $S\not\in\{r_7(q),r_{14}(q)\}^*$. Therefor $G\not\in \{r_7(q),r_{14}(q)\}^*$, a contradiction with Lemma \ref{good2}. If $m=6$ or $m=4$ then $|Aut(S)|_{(\{2,p\}\cup R_4(q))'}<|L|_{(\{2,p\}\cup R_4(q))'}$, a contradiction.

Assume that $S\simeq\ ^2D_m(p^l)$. We have $a=2lm$, $|S|_p=p^{lm(m-1)}=p^{a(m-1)/2} \leq p^{4a}$. Thus, $m\leq9$. From the facts that $|S|\leq |L|$ and $|L|_{(\{2,p\}\cup R_4(q))'}\leq|Aut(S)|_{(\{2,p\}\cup R_4(q))'}$, it is easy to get a contradiction.

Assume that $S\simeq\ ^{\varepsilon}E_{\alpha}(p^l)$, where $\alpha \in \{6,7,8\}$, if $\alpha=6$ then $\varepsilon \in \{+,-\}$ otherwise $\alpha$ is the empty symbol. In this case it is easy to show that $|S|_p>|L|_p$; a contradiction.
\end{proof}

\begin{lem}
$G\simeq L$
\end{lem}
\begin{proof}
From Lemma \ref{Liner8} it follows that $S\simeq L$. Assume that $K$ is not trivial. From Lemmas \ref{factorD8}, \ref{R8}, and \ref{MnogoR} it follows that $K$ is a $\{2,p\}\cup R_4(q)$-group. Take $h\in S$ is an element of order $k_{14}(q)$. Since $r_{14}(q)$ is not adjacent to $p$ in the graph $GK(S)$, we obtain $(Ind_S(h))_p=|S|_p=|L|_p$. Assume that $p\in\pi(K)$. Let $\widetilde{G}=G/O_{p'}(G)$, $h'\in \widetilde{G}$ be the preimage of the element $h$. If $h'$ acts non trivially on $O_p(\widetilde{G})$ then $(Ind_{\widetilde{G}}(h'))_p>|L|_p$; a contradiction. Thus $h'$ acts trivially on $O_p(\widetilde{G})$. Hence any element conjugate to $h'$ acts trivially on $O_p(\widetilde{G})$. Since $S$ is a simple group, the minimal preimage $H<\widetilde{G}$ of $S$ is contained in $C_G(O_p(\widetilde{G}))$. Hence for any $x\in O_p(\widetilde{G})$ we have $(Ind_{\widetilde{G}}(x))_{R_3(q)\cup R_6(q)}=1$. Since $p$ is not divisible $O_{p'}(G)$, we obtain that $G$ contain $x'$ such that $(Ind_{G}(x'))_{R_3(q)\cup R_6(q)}=1$; a contradiction with Lemma \ref{GorBig}. Thus $p\not\in\pi(K)$. Similarly, we can show that $K$ is trivial.

From Lemma \ref{Order} we obtain $S=G$.
\end{proof}

\end{document}